\documentclass[reqno,a4paper]{amsart}

\usepackage{amsmath,amssymb, amsthm}
\usepackage{dsfont}
\usepackage{color}

\makeatletter
\@addtoreset{equation}{section}
\makeatother
\makeatletter

\newcommand{\R}{{\mathds R}}

\newcommand{\N}{{\mathds N}}
\newcommand{\eps}{\varepsilon}

\newtheorem{definition}{Definition}[section]
\newtheorem{lemma}[definition]{Lemma}
\newtheorem{theorem}[definition]{Theorem}
\newtheorem{proposition}[definition]{Proposition}
\newtheorem{remark}[definition]{Remark}
\newtheorem{corollary}[definition]{Corollary}

\title[Unbounded diffusion and inverse square potentials]{Elliptic operators with unbounded diffusion coefficients perturbed by inverse square potentials in $L^p$--spaces}
\author{S. Fornaro}
\address{Dipartimento di Matematica ``F. Casorati'', Universit\`a degli studi di Pavia, via Ferrata, 1, 27100, Pavia, Italy.}
\email{simona.fornaro@unipv.it}
\author{F. Gregorio}
\author{A. Rhandi}
\thanks{The authors are members of the Gruppo Nazionale per l'Analisi Matematica,
la Probabilit\`a e le loro Applicazioni (GNAMPA) of the Istituto Nazionale di Alta Matematica
(INdAM)}
\address{Dipartimento di Ingegneria dell'Informazione, Ingegneria Elettrica e Matematica Applicata, Universit\`a di Salerno, Via Ponte Don Melillo, 84084 FISCIANO (Sa), Italy.}
\email{fgregorio@unisa.it}
\email{arhandi@unisa.it}
\keywords{Inverse square potential, positivity preserving $C_0$-semigroup, core, dissipative and dispersive operator, Hardy's inequality, unbounded diffusion}
\subjclass[2000]{35P05, 35J70, 35K65}

\date{}
\begin{document}

\maketitle

\begin{abstract}
In this paper we give sufficient conditions on $\alpha \ge 0$ and $c\in \R$ ensuring that the space of test functions $C_c^\infty(\R^N)$ is a core for the operator
$$L_0u=(1+|x|^\alpha )\Delta u+\frac{c}{|x|^2}u=:Lu+\frac{c}{|x|^2}u,$$
and $L_0$ with a suitable domain generates a quasi-contractive and positivity preserving $C_0$-semigroup in $L^p(\R^N),\,1<p<\infty$. The proofs are based on some $L^p$-weighted Hardy's inequality and perturbation techniques.
\end{abstract}



\section{Introduction}
\noindent
Let us consider the elliptic operator
\begin{equation*} \label{op-L}
L=(1+|x|^\alpha)\Delta,
\end{equation*}
where $\alpha\ge 0$. In this paper we want to study the perturbation of $L$ with a singular potential. More precisely, we consider the operator
\begin{equation*} \label{new-op}
L_0=L+\frac{c}{|x|^2}
\end{equation*}
and we look for optimal conditions on $c \in\R$ and $\alpha$ ensuring that $L_0$ with a suitable domain generates a positivity preserving $C_0$-semigroup  in $L^p(\R^N)$.

Let us recall first some known results for Schr\"odinger operators with inverse-square potentials.\\
It is known, see \cite[Theorem 2]{S}, that the realization $A_2$ of the Schr\"odinger operator $\mathcal{A}=\Delta +c|x|^{-2}$ in $L^2(\R^N)$ is essentially selfadjoint on $C_c^\infty(\R^N\setminus \{0\})$ if and only if $$c\le \frac{(N-2)^2}{4}-1=:c_0,$$
cf. \cite[Proposition VII.4.1]{EE} or \cite[Theorem X.11]{RS}, when $N\ge 5$.
\\
The characterization of the existence of positive weak solutions to the parabolic problem associated with the operator $\mathcal{A}$ was first discovered by Baras and Goldstein \cite{BG}, where they proved that a positive weak solution exists if and only if $c\le \frac{(N-2)^2}{4}$.

Using perturbation techniques it is proved in \cite[Theorem 6.8]{Oka82} that $A_2$ is selfadjoint provided that $c<c_0$. These techniques were generalized to the $L^p$-setting, $1<p<\infty$, and it is obtained that $A_p$, the realization of $\mathcal{A}$ in $L^p(\R^N)$, with domain $W^{2,p}(\R^N)$ generates a contractive and positive $C_0$-semigroup in $L^p(\R^N)$, and $C_c^\infty(\R^N)$ is a core for $A_p$, if $N>2p$ and $$c<\frac{(p-1)(N-2p)N}{p^2}=:\beta_0,$$
see \cite[Theorem 3.11]{Oka96}. In the case where $N\le 2p$, it is proved that $A_p$ with domain $D(A_p)=W^{2,p}(\R^N)\cap \{u\in L^p(\R^N);\,|x|^{-2}u\in L^p(\R^N)\}$ is $m$-sectorial if $c<\beta_0$, see \cite[Theorem 3.6]{Oka96}.

If one replaces the Laplacian by the Ornstein-Uhlenbeck operator similar results were obtained recently in \cite{DR, FR}.

In this paper we obtain similar results as in \cite[Theorem 3.11]{Oka96} when replacing $\Delta$ by $L$. We discuss also the generation of a $C_0$-semigroup of the operator $(1+|x|^\alpha)\Delta -\eta |x|^\beta +\frac{c}{|x|^2}$, where $\eta$ is a positive constant, $\alpha \ge 2$ and $\beta >\alpha -2$.

Now, let us recall some definitions. An operator $(A,D(A))$ on a Banach space $X$ is called accretive if $-A$ is dissipative. It is $m$--accretive if $A$ is accretive and $X=R(\lambda +A)$, the range of the operator $(\lambda +A)$. An accretive operator $(A,D(A))$ is called essentially $m$--accretive if its closure $\overline{A}$ is $m$--accretive.

Our approach relies on the following perturbation result due to N. Okazawa, see \cite[Theorem 1.7]{Oka96}.


\begin{theorem}\label{Oka-th1.7}
Let $A$ and $B$ be linear $m$--accretive operators in $L^p(\R^N)$, with $p\in (1,+\infty)$. Let $D$ be a core of $A$.
 Assume that
\begin{itemize}
\item[(i)] there are constants $\tilde{c}, a\ge 0$ and $k_1>0$ such that for all $u\in D$ and $\eps >0$
$$
{\mathcal Re}\langle Au, \|B_\eps u\|_p^{2-p}|B_\eps u|^{p-2}B_\eps u\rangle \ge k_1\|B_\eps u\|_p^2 -\tilde{c}\|u\|_p^2-a\|B_\eps u\|_p\|u\|_p
$$
where $B_\eps$ denote the Yosida approximation of $B$;
\item[(ii)]  ${\mathcal Re}\langle u, \|B_\eps u\|_p^{2-p}|B_\eps u|^{p-2}B_\eps u\rangle \ge 0$, for all $u \in L^p(\R^N)$ and $\eps >0$;
\item[(iii)] there is $k_2>0$ such that $A-k_2 B$ is accretive.
\end{itemize}
Set $k=\min\{k_1,k_2\}$. If $c>-k$ then $A+cB$ with domain $D(A+cB)=D(A)$ is $m$--accretive and any core of $A$ is also a core for $A+cB$. Furthermore, $A-kB$ is essentially $m$--accretive on $D(A)$.
\end{theorem}

In order to apply the above theorem, we need some preliminary results on the operator $L$ and some Hardy's inequalities.


\section{Preliminary results}
Let us begin with the generation results for suitable realizations $L_p$ of the operator $L$ in $L^p(\R^N),\,1<p<\infty$. Such results have been proved in \cite{FL,LR,Met_Spi}.
More specifically,
the case $\alpha\le 2$ has been investigated in \cite{FL} for $1<\alpha \le 2$ and in \cite{LR} for $\alpha \le 1$, where the authors proved the following result.
\begin{theorem}\label{thm-luca-simona}
If $\alpha \in [0,2]$ then, for any $p \in(1,+\infty)$, the realization $L_p$ of $L$ with domain
$$
D_p=\{u\in W^{2,p}(\R^N) : |x|^\alpha |D^2 u|, |x|^{\alpha/2}|\nabla u|\in L^p(\R^N) \}
$$
generates a positive and strongly continuous analytic semigroup. Moreover $C_c^\infty(\R^N)$ is a core for $L_p$.
\end{theorem}

The case $\alpha >2$ is more involved and is studied in \cite{Met_Spi}, where the following facts are established.
\begin{theorem}\label{thm-giorgio-chiara}
Assume that $\alpha >2$.
\begin{itemize}
\item[1.] If $N=1,2$, no realization of $L$ in $L^p(\R^N)$ generates a strongly continuous (resp. analytic) semigroup.
\item[2.] The same happens if $N\ge 3$ and $p\le N/(N-2)$.
\item[3.] If $N\ge 3$, $p > N/(N-2)$ and $2 <\alpha\le(p-1)(N -2)$, then the maximal realization $L_p$ of the operator $L$ in $L^p(\R^N)$ with the maximal domain
$$D_{\rm max}=\{ u\in W^{2,p}(\R^N) : (1+|x|^\alpha) \Delta u\in L^p(\R^N)\}$$
generates a positive $C_0$-semigroup of contractions, which is also analytic if $\alpha< (p - 1)(N - 2)$.
%
\item[4.] If $N\ge 3$, $p > N/(N-2)$ and $2<\alpha<\frac{N(p-1)}{p}$ the domain $D_{\rm max}$ coincides with the space
$$
\widehat{D_p}=\{u\in W^{2,p}(\R^N) : |x|^{\alpha-2}u, |x|^{\alpha-1}|\nabla u|, |x|^{\alpha}|D^2 u|\in L^p(\R^N)\}.
$$
Moreover, $C_c^\infty(\R^N)$ is a core for $L$.
\end{itemize}
\end{theorem}
If we consider the operator $\tilde{L}:=L-\eta |x|^\beta$ with $\eta >0$ and $\beta >\alpha -2$ then we can drop the above conditions on $p,\,\alpha$ and $N$, as the following result shows, see \cite{CRT}, where the quasi-contractivity can be deduced from the proof of Theorem 4.5 in \cite{CRT}.
\begin{theorem}\label{thm:CRT}
Assume $N\ge 3$. If $\alpha >2$ and $\beta >\alpha -2$ then, for any $p\in (1,\infty)$, the realization $\tilde{L}_p$ of $\tilde{L}$ with domain
$$\widetilde{D_p}=\{u\in W^{2,p}(\R^N) : |x|^{\beta}u, |x|^{\alpha-1}|\nabla u|, |x|^{\alpha}|D^2 u|\in L^p(\R^N)\}$$
generates a positive and strongly continuous quasi-contractive analytic semigroup. Moreover, $C_c^\infty(\R^N)$ is a core for $\tilde{L}_p$.
\end{theorem}

From now on we assume $N\ge 3,\,\alpha \ge 0$. We set
\begin{equation}       \label{hardy_const}
\gamma_\alpha=\left(\frac{N+\alpha-2}{p}\right)^2
\end{equation}
and recall the following Hardy's inequality. For a proof we refer to \cite[Lemma 2.2 \& Lemma 2.3]{Oka96} for the case $\alpha=0$ and \cite[Appendix]{Met_Spi} for $\alpha \ge 2$. Here we give a simple proof based on the method of vector fields introduced by Mitidieri in \cite{Mi}, which holds for any $\alpha \ge0$.

\begin{lemma}       \label{hardy-p}
For every $u\in W^{1,p}(\R^N)$ with compact support, one has
\begin{equation}         \label{hardy-p-formula2}
\gamma_\alpha \int_{\R^N}\frac{|u|^p}{|x|^2}|x|^\alpha dx \le  \int_{\R^N} |\nabla u|^{2}\, |u|^{p-2}|x|^\alpha dx.
\end{equation}
The inequality holds true even if $u$ is replaced by $|u|$.
\end{lemma}
\begin{proof}
By density, it suffices to prove \eqref{hardy-p-formula2} for $u\in C_c^1(\R^N)$. So, for every $\lambda\ge 0$, let us consider the vector field $F(x)=\lambda \frac{x}{|x|^2}|x|^\alpha,\,x\neq 0,$ and set $d\mu(x)=|x|^\alpha dx$.
Integrating by parts and applying H\"older and Young's inequalities we get
\begin{align*}
\int_{\R^N}|u|^p {\rm div}F\,dx &= \lambda(N-2+\alpha)\int_{\R^N}\frac{|u|^p}{|x|^2}\,d\mu \\
&= -p\lambda \int_{\R^N}|u|^{p-2} {\mathcal Re }(u\nabla \bar u)\cdot \frac{x}{|x|^2}\,d\mu \\
&\le p\lambda \left(\int_{\R^N} |\nabla u|^{2}\, |u|^{p-2}\,d\mu \right)^{\frac{1}{2}}\left(\int_{\R^N}\frac{|u|^p}{|x|^2}\,d\mu \right)^{\frac{1}{2}}\\
&\le \int_{\R^N} |\nabla u|^{2}\, |u|^{p-2}\,d\mu +\frac {\lambda^2p^2}{4}\int_{\R^N}\frac{|u|^p}{|x|^2}\,d\mu.
\end{align*}
In the computations above, we used the identity $\nabla |u|^p=p|u|^{p-2}{\mathcal Re}(u\nabla\bar u)$.
Hence,
$$\left[\lambda (N-2+\alpha)-\frac {\lambda^2p^2}{4}\right]\int_{\R^N}\frac{|u|^p}{|x|^2}|x|^\alpha \,dx \le \int_{\R^N} |\nabla u|^{2}\, |u|^{p-2}|x|^\alpha \,dx .$$
By taking the maximum over $\lambda$ of the function $\psi(\lambda)=\lambda(N-2+\alpha)-\lambda^2p^2/4$, we get \eqref{hardy-p-formula2}.

We note here that the integration by parts is straightforward when $p\ge 2$. For $1<p<2$, $|u|^{p-2}$ becomes singular near the zeros of $u$. Also in this case the integration by parts is allowed, see \cite{MS}.

By using the identity $\nabla | u|^p=p|u|^{p-1}\nabla |u|$ in the computations above, the statement holds with $u$ replaced by $|u|$.
\end{proof}

\begin{remark}         \label{hardy-best-const}
{\rm The constant $\gamma_\alpha$ in  \eqref{hardy-p-formula2} is optimal, as shown in \cite[Appendix]{Met_Spi}}.
\end{remark}
\begin{remark}         \label{hardy-cor}
{\rm Hardy's inequality \eqref{hardy-p-formula2} holds even if $u$ is replaced by $u_+:=\sup(u,0)$, since  $u_+\in W^{1,p}(\R^N)$, whenever  $u\in W^{1,p}(\R^N)$ (cf. \cite[Lemma 7.6]{GT})}.
\end{remark}

As a consequence of Lemma \ref{hardy-p} we have the following results.
\begin{proposition}       \label{prelim1}
Assume $ \alpha \le (N-2)(p-1)$. Let $V\in L^p_{\rm loc}(\R^N\setminus \{0\})$. If $V(x)\le \frac{c}{|x|^2}$, $x\neq 0$, with $c\le (p-1)\gamma_0$, then $L+V$ with domain $C_c^\infty(\R^N)$ is dissipative in $L^p(\R^N)$.
\end{proposition}
\begin{proof}
Let $u\in C_c^\infty(\R^N)$. Take $\delta >0$ if $1<p<2$ and $\delta =0$ if $p\ge 2$. Then we have
\begin{align*}
\langle Lu,u(|u|^2+\delta)^{\frac{p-2}{2}}\rangle &= -\int_{\R^N}\nabla u\cdot \nabla\left(\overline{u}
(|u|^2+\delta)^{\frac{p-2}{2}}\right)(1+|x|^\alpha) d x \\
&\quad - \alpha\int_{\R^N}\overline{u}  (|u|^2+\delta)^\frac{p-2}{2}\nabla u\cdot x |x|^{\alpha-2} dx\\
&=- \int_{\R^N}(|u|^2+\delta)^{\frac{p-2}{2}}|\nabla u|^2(1+|x|^\alpha) d x \\
&\quad -(p-2)\int_{\R^N}(|u|^2+\delta)^{\frac{p-4}{2}}(|u|\nabla |u|)\cdot (\overline{u}\nabla u)(1+|x|^\alpha) d x\\
&\quad - \alpha\int_{\R^N}\overline{u}  (|u|^2+\delta)^\frac{p-2}{2}\nabla u\cdot x |x|^{\alpha-2} dx.
\end{align*}
So, using the identities $|\nabla |u||^2\le |\nabla u|^2$ and $|u|\nabla |u|={\mathcal Re}(\overline{u}\nabla u)$, we obtain
\begin{align*}
{\mathcal Re}\langle Lu ,u|u|^{p-2}\rangle& \le  -(p-1)\int_{\R^N}|\nabla |u||^2 |u|^{p-2}(1+|x|^\alpha) d x\\
&\quad - \alpha\int_{\R^N}|u|^{p-1}\nabla |u| \cdot x |x|^{\alpha-2} dx
\end{align*}
if $p\ge 2$. The case $1<p<2$ can be handled similarly.
Thus, by H\"older's inequality we have
\begin{align}\label{estim-L}
{\mathcal Re}\langle (L+V)u,u|u|^{p-2}\rangle &\le  -(p-1)\int_{\R^N}|\nabla |u||^2 |u|^{p-2}(1+|x|^\alpha) d x  +\int_{\R^N}V|u|^p d x \nonumber \\
&\quad+\alpha  \bigg(\int_{\R^N}|\nabla |u||^2 |u|^{p-2}|x|^\alpha dx\bigg)^\frac 1 2 \bigg( \int_{\R^N}\frac{|u|^p}{|x|^2} |x|^\alpha dx\bigg)^\frac{1}{2}.
\end{align}
Set
$$
\begin{array}{ll}
\displaystyle I_\alpha^2=\int_{\R^N}|\nabla |u||^2 |u|^{p-2}|x|^\alpha d x, & \qquad \displaystyle J_\alpha^2=   \int_{\R^N}\frac{|u|^p}{|x|^2} |x|^\alpha dx\\[4mm]
\displaystyle I_0^2=\int_{\R^N}|\nabla |u||^2 |u|^{p-2} d x , & \qquad  \displaystyle J_0^2=\int_{\R^N}\frac{|u|^p}{|x|^2}  dx.
\end{array}
$$
Taking the assumption on $V$ into account we obtain
 \begin{align*}
 {\mathcal Re}\langle (L+V)u ,u|u|^{p-2}\rangle& \le -(p-1)I_0^2-(p-1) I_\alpha^2+ c\,J_0^2+\alpha I_\alpha J_\alpha.
 \end{align*}
 Since $c\le (p-1)\gamma_0$ and Lemma \ref{hardy-p} holds for $\alpha=0$, we have that $-(p-1)I_0^2+c\,J_0^2\le 0$. Now, the inequality
 $$-(p-1)I_\alpha^2+\alpha I_\alpha J_\alpha \le 0$$ holds true if $$-(p-1)+\alpha \gamma_\alpha^{-1/2}\le 0,$$
 thanks again to Lemma \ref{hardy-p}. The latter inequality is equivalent to $\alpha\le (N-2)(p-1)$, which is the assumption.
This ends the proof.
 \end{proof}

\begin{remark}
{\rm The assumption $\alpha \le (N-2)(p-1)$ is optimal for the dissipativity of $L$, as proved in \cite[Proposition 8.2]{Met_Spi}.
}
\end{remark}

Hence, in order to apply Theorem \ref{Oka-th1.7}, we have established
\begin{corollary}\label{diss}
Assume $ \alpha \le (N-2)(p-1)$. Then, the operator $\displaystyle L+\frac{c}{|x|^2}$ with $c\le (p-1)\gamma_0$ and domain $C_c^\infty(\R^N)$ is dissipative in $L^p(\R^N)$.
\end{corollary}

Let us recall the definition of dispersivity of an operator. A (real) linear operator $A$ with domain $D(A)$ in $L^p(\R^N)$ is called dispersive if $$\langle Au, u_+^{p-1}\rangle\le 0 \quad \hbox{\ for all }u\in D(A).$$
For more details on dispersive operators we refer to \cite[C-II.1]{Na86}.
\begin{proposition}           \label{dispersivity}
Assume $\alpha \le (N-2)(p-1)$. Then, the operator $L+\frac{c}{|x|^2}$ with $c\le (p-1)\gamma_0$ and domain $C_c^\infty(\R^N)$ is dispersive in $L^p(\R^N)$.
\end{proposition}
\begin{proof}
Let $u\in C_c^\infty(\R^N)$ be real-valued and fix $\delta>0$. Replacing $u$ by $u_+$ in the proof of Proposition \ref{prelim1} and since $u_+\in W^{1,p}(\R^N)$, we deduce that
\begin{align*}
\langle Lu, u_+(u_+^2+\delta)^{\frac{p-2}{2}}\rangle &=- \int_{\R^N}(u_+^2+\delta)^{\frac{p-2}{2}}|\nabla u_+|^2(1+|x|^\alpha) d x \\
&\quad -(p-2)\int_{\R^N}(u_+^2+\delta)^{\frac{p-4}{2}}u_+^2|\nabla u_+|^2(1+|x|^\alpha) d x\\
&\quad - \alpha\int_{\R^N} u_+ (u_+^2+\delta)^\frac{p-2}{2}\nabla u_+\cdot x |x|^{\alpha-2} dx.
\end{align*}
Then,
\begin{align*}
\langle Lu, u_+(u_+^2+\delta)^{\frac{p-2}{2}}\rangle &\le (1-p)\int_{\R^N}(u_+^2+\delta)^{\frac{p-4}{2}}u_+^2|\nabla u_+|^2(1+|x|^\alpha) d x \\
&\quad -\alpha \int_{\R^N} u_+ (u_+^2+\delta)^\frac{p-2}{2}\nabla u_+\cdot x |x|^{\alpha-2} dx,
\end{align*}
where here we take $\delta=0$ if $p\ge 2$ and $\delta >0$ if $1<p<2$. Thus, letting $\delta \to 0$ if $1<p<2$, and applying H\"older's inequality we obtain
$$
\langle (L+\frac{c}{|x|^2})u, u_+^{p-1}\rangle \le (1-p)I_{0,+}^2+(1-p)I_{\alpha ,+}^2+c J_{0,+}^2+\alpha I_{\alpha ,+}J_{\alpha ,+},$$
where
$$
\begin{array}{ll}
\displaystyle I_{\alpha ,+}^2=\int_{\R^N}|\nabla u_+|^2 u_+^{p-2}|x|^\alpha d x, & \qquad \displaystyle J_{\alpha ,+}^2=   \int_{\R^N}\frac{u_+^p}{|x|^2} |x|^\alpha dx\\[4mm]
\displaystyle I_{0,+}^2=\int_{\R^N}|\nabla u_+|^2 u_+^{p-2} d x , & \qquad  \displaystyle J_{0,+}^2=\int_{\R^N}\frac{u_+^p}{|x|^2}  dx.
\end{array}
$$
As in the proof of Proposition \ref{prelim1}, the assertion follows now by Lemma \ref{hardy-p} and Remark \ref{hardy-cor}.
\end{proof}

The next proposition deals with the operator $L-\eta|x|^\beta$.
\begin{proposition}\label{p:tilde}
Let $V\in L^p_{loc}(\R^N \setminus \{0\})$ with $V\le \frac{c}{|x|^2},\,x\neq 0$ and $c\le (p-1)\gamma_0$. Set $\widetilde{L}=L-\eta|x|^\beta$.
\begin{itemize}
\item[\em (i)] If $\alpha \ge 2$, $\beta >\alpha -2$ and $\eta >0$ then the operator $\widetilde{L}+V$ with domain $C_c^\infty(\R^N)$ is quasi-dissipative in $L^p(\R^N)$.
\item[\em (ii)] If $0\le\alpha \le (N-2)(p-1)$, $\beta=\alpha-2$ then $\widetilde{L}+V$ with domain $C_c^\infty(\R^N)$ is dissipative in $L^p(\R^N)$ if
\begin{equation} \label{cond_eta}
\eta+\frac{(N+\alpha-2)^2}{pp'}-\frac{\alpha(N+\alpha-2)}{p}\ge 0.
\end{equation}
\end{itemize}
\end{proposition}
\begin{proof}
{\em (i)} If $\beta > \alpha-2$, applying \eqref{estim-L} and Young's inequality we obtain
\begin{align*}
{\mathcal Re}\langle (\widetilde{L}+V)u,u|u|^{p-2}\rangle &\le -(p-1)I_0^2-(p-1)I_\alpha^2+cJ_0^2+\eps I_\alpha^2\\
&\quad +\int_{\R^N}\left(\frac{\alpha^2}{4\eps}|x|^{\alpha -2}-\eta |x|^\beta\right)|u|^p\,dx\\
&\le -(p-1)I_0^2-(p-1-\eps)I_\alpha^2+cJ_0^2+M\|u\|^p_p
\end{align*}
for $u\in C_c^\infty(\R^N)$ and any $\eps >0$, where $M$ is a positive constant such that $\frac{\alpha^2}{4\eps}|x|^{\alpha -2}-\eta |x|^\beta \le M$ for all $x\in \R^N$, which holds since $\beta >\alpha -2\ge 0$. Choosing now $\eps \le p-1$ and applying \eqref{hardy-p-formula2} we obtain
$${\mathcal Re}\langle (\widetilde{L}+V)u,u|u|^{p-2}\rangle \le M\|u\|^p_p$$
which means that $\widetilde{L}+V$ with domain $C_c^\infty(\R^N)$ is quasi-dissipative in $L^p(\R^N)$.\\
{\em (ii)} If $\beta=\alpha-2$ then \eqref{estim-L} gives
\begin{equation*}
{\mathcal Re}\langle (\widetilde{L}+V)u,u|u|^{p-2}\rangle \le -(p-1)I_0^2-(p-1)I_\alpha^2+cJ_0^2+\alpha I_\alpha J_\alpha -\eta J_\alpha^2.
\end{equation*}
If $\eta\ge 0$, then the conclusion easily follows as in the end of the proof of Proposition \ref{prelim1}, under the assumption $c\le \gamma_0(p-1)$ and $\alpha \le (N-2)(p-1)$. If $\eta <0$ then by Lemma \ref{hardy-p} we have
$$
-(p-1)I_\alpha^2+\alpha I_\alpha J_\alpha -\eta J_\alpha^2\le \big(-(p-1)+\alpha \gamma_\alpha^{-1/2} -\eta\gamma_\alpha^{-1}\big) I_\alpha^2.
$$
The right hand side is nonpositive if
$$
\eta+\frac{(N+\alpha-2)^2}{pp'}-\frac{\alpha(N+\alpha-2)}{p}\ge 0.
$$
\end{proof}

\begin{remark}
{\rm Condition \eqref{cond_eta} is sharp as proved in \cite[Proposition 4.2]{MOSS}.}
\end{remark}

\section{Main results}
In this section we state and prove the main results of this paper.\\
In order to apply Theorem \ref{Oka-th1.7} to our situation we need the following lemma whose proof follows the same lines of \cite[Lemma 3.4]{Oka96}.
\begin{lemma}\label{main-lem}
Set $ V_\eps=\frac{1}{|x|^2+\eps}$, $\eps>0$. Assume $\alpha \le (N-2)(p-1)$. Then for every $u\in C_c^\infty(\R^N)$
\begin{equation}\label{main-inequality}
{\mathcal Re}\langle -Lu,|V_\eps u|^{p-2}V_\eps u\rangle \ge \beta_0\int_{\R^N}V_\eps^p|u|^p\,dx+\beta_\alpha \int_{\R^N}V_\eps^p|u|^p|x|^\alpha\,dx,
\end{equation}
where
$$\beta_0=\frac{N(p-1)(N-2p)}{p^2},\qquad \beta_\alpha =\frac{(Np-N-\alpha)(N+\alpha -2p)}{p^2}.$$
Moreover, if $N>2p$ then both $\beta_0$ and $\beta_\alpha$ are positive.
\end{lemma}
\begin{proof}
Let $u\in C_c^\infty(\R^N)$ and set $u_\delta=\big((R|u|)^2+\delta\big)^\frac 1 2$, where $R^p:=V_\eps^{p-1}$.  In the computations below, we have to take $\delta>0$ in the case $1<p<2$, whereas we only take $\delta=0$ to deal with the case $p\ge 2$. We have
$$
\langle -Lu, |V_\eps u|^{p-2}V_\eps u\rangle=-\lim_{\delta\to 0}\int_{\R^N}u_\delta^{p-2}R^2 \bar u\, Lu \,d x.
$$
Integrating by parts we have
\begin{align} \label{prima}
-\int_{\R^N}u_\delta^{p-2}R^2 \bar u\, Lu \,dx & = \int_{\R^N}R^2  \bar u\, \nabla u\cdot\nabla(u_\delta^{p-2})  (1+|x|^\alpha) dx\nonumber\\
&\quad +\int_{\R^N}u_\delta^{p-2}\nabla u\cdot\nabla(R^2 \bar u ) (1+|x|^\alpha) \,dx\\
 & \quad + \alpha\int_{\R^N} u_\delta^{p-2}R^2 \overline{u}\, |x|^{\alpha-2} x\cdot \nabla u \,dx.\nonumber
\end{align}
Now, computing $\nabla(u_\delta^{p-2})$ and writing $R^2  \bar u\, \nabla u=R\bar u\big(\nabla(Ru)-u\nabla R\big)$ we have
\begin{align*}
 \int_{\R^N}R^2  &\bar u\, \nabla u\cdot\nabla(u_\delta^{p-2})  (1+|x|^\alpha) dx\\
&=   \frac{p-2}{2}\int_{\R^N}u_\delta^{p-4}R\bar u\,\nabla (R^2|u|^2)\cdot\nabla (R\,u)   (1+|x|^\alpha) dx\\
& \quad - \frac{p-2}{2}\int_{\R^N}u_\delta^{p-4}R|u|^2\nabla (R^2|u|^2)\cdot\nabla  R\,  (1+|x|^\alpha) dx.
\end{align*}
Using also the identity
$$
\nabla(R^2 \bar u )\cdot \nabla u =|\nabla (R \, u )|^2-u\nabla (R\, \bar u )\cdot \nabla R+R\, \bar u \nabla R \cdot \nabla u
$$
Equation \eqref{prima} yields
\begin{align*}
-\int_{\R^N}u_\delta^{p-2}R^2 \bar u\, Lu \,dx &=   \frac{p-2}{2}\int_{\R^N}u_\delta^{p-4}R\bar u\,\nabla (R^2|u|^2)\cdot\nabla (R\,u)   (1+|x|^\alpha) dx\\
 & \quad+\int_{\R^N}u_\delta^{p-2}|\nabla (R \, u )|^2 (1+|x|^\alpha) \,dx\\
& \quad \underbrace{- \frac{p-2}{2}\int_{\R^N}u_\delta^{p-4}R|u|^2\nabla (R^2|u|^2)\cdot\nabla  R\,  (1+|x|^\alpha) dx}_{=I}\\
 & \quad+\underbrace{ \int_{\R^N}u_\delta^{p-2}R\, \bar u \nabla R \cdot \nabla u  (1+|x|^\alpha) \,dx}_{=J}\\
 & \quad\underbrace{-\int_{\R^N}u_\delta^{p-2}u\nabla (R\, \bar u )\cdot \nabla R  (1+|x|^\alpha) \,dx}_{=K}\\
 & \quad +\alpha\int_{\R^N} u_\delta^{p-2}R^2 \overline{u}\, |x|^{\alpha-2} x\cdot \nabla u \,dx.
\end{align*}
Now, introduce the function $Q=R^p$. Writing $\nabla(R^2|u|^2)=2R|u|^2\nabla R+2|u|R^2\nabla|u|$ we have
\begin{align*}
I&=-(p-2)\int_{\R^N}u_\delta^{p-4}R^2\,| u|^4\,|\nabla R|^2   (1+|x|^\alpha) dx\\
&\quad -(p-2)\int_{\R^N}u_\delta^{p-4}|u|^3\,R^3\nabla R\cdot\nabla| u| \,  (1+|x|^\alpha) dx\\
&=-\frac{p-2}{p^2}\int_{\R^N}u_\delta^{p-4}R^{4-2p}\,| u|^4\,|\nabla Q|^2   (1+|x|^\alpha) dx\\
&\quad -\frac{p-2}{p}\int_{\R^N}u_\delta^{p-4}|u|^3\,R^{4-p}\nabla Q\cdot\nabla |u|\,  (1+|x|^\alpha) dx.
\end{align*}
Moreover,
\begin{align*}
J+K&=\int_{\R^N}u_\delta^{p-2}\Big(R\,\bar u\nabla R\cdot \nabla u-u\nabla (R\,\bar u) \cdot \nabla R\Big) (1+|x|^\alpha) \,dx\\
&=-\int_{\R^N}u_\delta^{p-2}|u|^2|\nabla R|^2 (1+|x|^\alpha) \,dx\\
&\quad+2i\int_{\R^N} u_\delta^{p-2}\,{\mathcal Im}(\bar u\,\nabla u)  \cdot R \nabla R(1+|x|^\alpha) \,dx\\
&=-\frac{1}{p^2}\int_{\R^N}u_\delta^{p-2}|u|^2 R^{2-2p}|\nabla Q|^2 (1+|x|^\alpha) \,dx\\
&\quad+\frac{2i}{p}\int_{\R^N} u_\delta^{p-2}R^{2-p}\,{\mathcal Im}(\bar u\,\nabla u)  \cdot  \nabla Q(1+|x|^\alpha) \,dx.
\end{align*}
Hence we have
\begin{align*}
-\int_{\R^N}u_\delta^{p-2}R^2 \bar u\, Lu \,dx &=   (p-2)\int_{\R^N}u_\delta^{p-4} R\, |u|\,\nabla (R\,|u|)\cdot (R\, \bar u) \nabla (R\,u)   (1+|x|^\alpha) dx\\
 & \quad+\int_{\R^N}u_\delta^{p-2}|\nabla (R \, u )|^2 (1+|x|^\alpha) \,dx+J_\delta\\
&\quad+\frac{2i}{p}\int_{\R^N} u_\delta^{p-2}R^{2-p}\,{\mathcal Im}(\bar u\,\nabla u)  \cdot  \nabla Q(1+|x|^\alpha) \,dx\\
& \quad +\alpha\int_{\R^N} u_\delta^{p-2}R^2 \overline{u}\, |x|^{\alpha-2} x\cdot \nabla u \,dx,
\end{align*}
where we have set
\begin{align*}
J_\delta=&-\frac{p-2}{p^2}\int_{\R^N}u_\delta^{p-4}R^{4-2p}\,| u|^4\,|\nabla Q|^2   (1+|x|^\alpha) dx\\
&\quad -\frac{p-2}{p}\int_{\R^N}u_\delta^{p-4}|u|^3\,R^{4-p}\nabla Q\cdot\nabla |u|\,  (1+|x|^\alpha) dx\\
 &\quad -\frac{1}{p^2}\int_{\R^N}u_\delta^{p-2}|u|^2 R^{2-2p}|\nabla Q|^2 (1+|x|^\alpha) \,dx. 
\end{align*}
Now, we take the real parts of both sides and apply the identity ${\mathcal Re}(\bar \phi \nabla \phi)=|\phi|\nabla|\phi|$ to obtain
\begin{align*}
-{\mathcal Re}\int_{\R^N}u_\delta^{p-2}R^2 \bar u\, Lu \,dx &=   (p-2)\int_{\R^N}u_\delta^{p-4}R^2| u|^2|\nabla (R|u|)|^2    (1+|x|^\alpha) dx\\
 & \quad+\int_{\R^N}u_\delta^{p-2}|\nabla (R \, u )|^2 (1+|x|^\alpha) \,dx+J_\delta\\
 &\quad + \alpha\int_{\R^N} u_\delta^{p-2}R^2 |u|\, |x|^{\alpha-2} x\cdot \nabla |u| \,dx\\
 &= (p-2)\int_{\R^N}u_\delta^{p-2}|\nabla (R|u|)|^2    (1+|x|^\alpha) dx\\
 &\quad - (p-2)\delta \int_{\R^N}u_\delta^{p-4} |\nabla (R|u|)|^2    (1+|x|^\alpha) dx\\
 &\quad + \int_{\R^N}u_\delta^{p-2}|\nabla (R \, u )|^2 (1+|x|^\alpha) \,dx+J_\delta\\
 &\quad + \alpha\int_{\R^N} u_\delta^{p-2}R^2 |u|\, |x|^{\alpha-2} x\cdot \nabla |u| \,dx.
\end{align*}
Since the inequality $|\nabla\phi |\ge |\nabla|\phi||$ holds and $\delta=0$ if $p\ge 2$, $\delta>0$ if $1<p<2$ we can estimate as follows
\begin{align*}
-{\mathcal Re}\int_{\R^N}u_\delta^{p-2}R^2 \bar u\, Lu \,dx &\ge    (p-1) \int_{\R^N}u_\delta^{p-2} |\nabla (R |u |)|^2    (1+|x|^\alpha) dx+J_\delta\\
 &\quad + \alpha\int_{\R^N} u_\delta^{p-2}R^2 |u|\, |x|^{\alpha-2} x\cdot \nabla |u| \,dx
\end{align*}
if $p\ge 2$ and
\begin{align*}
-{\mathcal Re}\int_{\R^N}u_\delta^{p-2}R^2 \bar u\, Lu \,dx &\ge    (p-1) \int_{\R^N}u_\delta^{p-2} |\nabla (R u )|^2    (1+|x|^\alpha) dx+J_\delta\\
 &\quad + \alpha\int_{\R^N} u_\delta^{p-2}R^2 |u|\, |x|^{\alpha-2} x\cdot \nabla |u| \,dx
\end{align*}
if $1<p<2$. Letting $\delta \to 0^+$, we are lead to
\begin{equation}
\begin{aligned}    \label{intermediate}
{\mathcal Re}\langle -Lu, |V_\eps u|^{p-2}V_\eps u\rangle
&\ge (p-1)\int_{\R^N}(R|u|)^{p-2}|\nabla(R |u|)|^2(1+|x|^\alpha) dx\\
&\quad-\frac{p-1}{p^2}\int_{\R^N}R^{-p}\,| u|^p\,|\nabla Q|^2   (1+|x|^\alpha) dx\\
&\quad -\frac{p-2}{p}\int_{\R^N}|u|^{p-1}\,\nabla Q\cdot\nabla |u|\,  (1+|x|^\alpha) dx\\
 &\quad + \alpha\int_{\R^N} |u|^{p-1}R^p  |x|^{\alpha-2} x\cdot \nabla |u| \,dx,
\end{aligned}
\end{equation}
where we have used again the inequality $|\nabla\phi |\ge |\nabla|\phi||$ in the first integral of the right hand side of \eqref{intermediate}, since for $1<p<2$ we had $|\nabla R\,u|^2$ instead of $|\nabla (R|u|)|^2$.
Now, by the identity $p|u|^{p-1} \nabla |u|=\nabla |u|^p$, integrating by parts and recalling the definition of $R$ we infer
\begin{align*}
&\quad -\frac{p-2}{p}\int_{\R^N}|u|^{p-1}\,\nabla Q\cdot\nabla |u|\,  (1+|x|^\alpha) dx\\
&= \frac{p-2}{p^2}\int_{\R^N} |u|^{p} \, \Delta R^p (1+|x|^\alpha) dx+\frac{\alpha(p-2)}{p^2}\int_{\R^N}|u|^{p}\nabla R^p\cdot x |x|^{\alpha-2} dx\\
&=- \frac{2N(p-1)(p-2)}{p^2}\int_{\R^N} V_\eps^p|u|^p(1+|x|^\alpha) dx\\
&\quad+\frac{4p(p-1)(p-2)}{p^2}\int_{\R^N}|x|^2V_\eps^{p+1}|u|^p (1+|x|^\alpha) dx\\
&\quad -\frac{2\alpha(p-1)(p-2)}{p^2}\int_{\R^N}V_\eps^p |u|^p |x|^\alpha dx
\end{align*}
and
\begin{align*}
&\alpha\int_{\R^N} |u|^{p-1}R^p  |x|^{\alpha-2} x\cdot \nabla |u| \,dx\\
&= - \frac{\alpha}{p}\int_{\R^N}|u|^p|x|^{\alpha-2} x\cdot \nabla R^pdx -\frac{\alpha(N+\alpha-2)}{p}\int_{\R^N} R^p |x|^{\alpha-2}|u|^pdx\\
&= \frac{2(p-1)\alpha}{p}\int_{\R^N} V_\eps^p|u|^p |x|^\alpha dx-\frac{\alpha(N+\alpha-2)}{p}\int_{\R^N}V_\eps^{p-1}|u|^p |x|^{\alpha-2}dx.
\end{align*}
Finally,
\begin{equation*}
\int_{\R^N}|u|^{p}R^{-p}|\nabla R^p|^2 (1+|x|^\alpha)\, dx=4(p-1)^2 \int_{\R^N}|x|^2V_\eps^{p+1}|u|^p (1+|x|^\alpha) dx.
\end{equation*}
By using such formulas in \eqref{intermediate} we obtain
\begin{align*}
{\mathcal Re}\langle -Lu, |V_\eps u|^{p-2}V_\eps u\rangle &\ge (p-1)\int_{\R^N}(R|u|)^{p-2}|\nabla(R |u|)|^2(1+|x|^\alpha) dx\\
&\quad - \frac{4(p-1)}{p^2}\int_{\R^N}|x|^2V_\eps^{p+1}|u|^p (1+|x|^\alpha) dx\\
&\quad -\frac{2N(p-1)(p-2)}{p^2}\int_{\R^N}V_\eps^{p}|u|^p (1+|x|^\alpha) dx\\
&\quad +\frac{4\alpha(p-1)}{p^2}\int_{\R^N}V_\eps^{p}|u|^p|x|^\alpha dx\\
&\quad -\frac{\alpha(N+\alpha-2)}{p}\int_{\R^N}V_\eps^{p-1}\frac{|u|^p}{|x|^2} |x|^\alpha dx.
\end{align*}
Applying Lemma \ref{hardy-p}, and using that $|x|^2 V_\eps\le 1$ we are lead to
\begin{equation} \label{stima-yosida}
\begin{aligned}
{\mathcal Re}\langle -Lu, &|V_\eps u|^{p-2}V_\eps u\rangle\ge (p-1)\gamma_0\int_{\R^N}\frac{V_\eps^{p-1}|u|^p}{|x|^2}dx\\
&\quad + \frac{N+\alpha-2}{p} \left(\frac{p-1}{p}(N+\alpha-2)-\alpha\right)\int_{\R^N}\frac{V_\eps^{p-1}|u|^p}{|x|^2}\,|x|^\alpha dx\\
	&\quad -\frac{p-1}{p^2}(4+2Np-4N)\int_{\R^N} V_\eps^{p}|u|^p (1+|x|^\alpha) dx\\
	&\quad +\frac{4\alpha(p-1)}{p^2}\int_{\R^N}V_\eps^p|u|^p |x|^\alpha dx.
\end{aligned}
\end{equation}
Since $\alpha\le (N-2)(p-1)$ we have $\frac{p-1}{p}(N+\alpha-2)-\alpha\ge0$ and then from the estimate $|x|^2 V_\eps\le 1$ it follows that
\begin{align*}
{\mathcal Re}\langle -Lu, |V_\eps u|^{p-2}V_\eps u\rangle &\ge (p-1)\gamma_0\int_{\R^N}V_\eps^{p}|u|^pdx\\
&\quad + \frac{N+\alpha-2}{p} \left(\frac{p-1}{p}(N+\alpha-2)-\alpha\right)\int_{\R^N}{V_\eps^{p}|u|^p}\,|x|^\alpha dx\\
	&\quad -\frac{p-1}{p^2}(4+2Np-4N)\int_{\R^N} V_\eps^{p}|u|^p (1+|x|^\alpha) dx\\
	&\quad +\frac{4\alpha(p-1)}{p^2}\int_{\R^N}V_\eps^p|u|^p |x|^\alpha dx.
\end{align*}
Thus we have
\begin{equation*}
{\mathcal Re}\langle -Lu, |V_\eps u|^{p-2}V_\eps u\rangle \ge \beta_0\int_{\R^N}V_\eps^{p}|u|^pdx+\beta_\alpha \int_{\R^N}{V_\eps^{p}|u|^p}\,|x|^\alpha dx,
\end{equation*}
where
\begin{align*}
\beta_0& =(p-1)\gamma_0-\frac{p-1}{p^2}(4+2Np-4N)=\frac{N(p-1)(N-2p)}{p^2}\\[3mm]
\beta_\alpha & =\frac{N+\alpha-2}{p} \left(\frac{p-1}{p}(N+\alpha-2)-\alpha\right)-\frac{p-1}{p^2}(4+2Np-4N)+\frac{4\alpha(p-1)}{p^2}\\
&=\frac{(Np-N-\alpha)(N+\alpha-2p)}{p^2}.
\end{align*}
So, if $N>2p$ then $\beta_0>0$ and since $0\le \alpha \le (N-2)(p-1)<N(p-1)$ we deduce that $\beta_\alpha >0$. 
\end{proof}

\begin{remark}\label{rem3.2}
{\rm We rewrite estimate \eqref{stima-yosida} as follows
$$
{\mathcal Re}\langle -Lu, |V_\eps u|^{p-2}V_\eps u\rangle \ge \beta_0\int_{\R^N}V_\eps^{p}|u|^p dx + \int_{\R^N}(k_0+k_1 V_\eps |x|^2) V_\eps^{p-1}|u|^p |x|^{\alpha-2} dx
$$
where
$$
\begin{array}{l}
k_0=\displaystyle\frac{N+\alpha-2}{p} \left(\frac{p-1}{p}(N+\alpha-2)-\alpha\right)\\[3mm]
k_1=	\displaystyle\frac{4\alpha(p-1)}{p^2}-\frac{p-1}{p^2}(4+2Np-4N).
\end{array}
$$
Notice that $k_0\ge 0$ if $\alpha\le (N-2)(p-1)$ and that $k_0+k_1=\beta_\alpha$. Now, $k_0+k_1 V_\eps |x|^2=f(|x|^2)$, where $f(r)=\frac{\eps k_0+(k_0+k_1)r}{\eps+r}$. Since $\displaystyle\inf_{[0,\infty)} f=\min\{k_0,k_0+k_1\}=:\mu$ we find
$$
{\mathcal Re}\langle -Lu-\mu|x|^{\alpha-2}u, |V_\eps u|^{p-2}V_\eps u\rangle \ge \beta_0\int_{\R^N}V_\eps^{p}|u|^p dx.
$$
The easiest case (see Lemma \ref{main-lem}) is when $\mu \ge 0$.}
\end{remark}

Now, we prove a similar estimate for the operator $\widetilde L=L-\eta|x|^\beta$.

\begin{lemma}\label{l:tilde}
Set $ V_\eps=\frac{1}{|x|^2+\eps}$, $\eps>0$. If $\beta>\alpha-2\ge 0$ and $\eta >0$ then for every $u\in C_c^\infty(\R^N)$
$$
{\mathcal Re}\langle -\widetilde Lu-mu, |V_\eps u|^{p-2}V_\eps u\rangle\ge \beta_0\int_{\R^N}V_\eps^{p}|u|^pdx+\delta_\alpha\int_{\R^N}V_\eps^p|u|^p |x|^\alpha dx,
$$
where $m=\min_{x\in\R^N}\left(\frac{N+\alpha-2}{p}\cdot\frac{(p-1)(N-2)-\alpha}{p} \,|x|^{\alpha-2}+\eta |x|^\beta\right) $, $\beta_0$ is given in Lemma \ref{main-lem} and
$$
\delta_\alpha=\frac{p-1}{p^2}(4\alpha-4-2Np+4N) .
$$
\end{lemma}
\begin{proof} We proceed as in the proof of Lemma \ref{main-lem}. From Remark \ref{rem3.2} and the inequality $|x|^2 V_\eps\le 1$ it follows that
\begin{align*}
&{\mathcal Re}\langle -\widetilde Lu, |V_\eps u|^{p-2}V_\eps u\rangle \\
&\ge \beta_0\int_{\R^N}V_\eps^{p}|u|^pdx\\
&\quad + \int_{\R^N}V_\eps^{p-1}|u|^p\left(\frac{N+\alpha-2}{p}\cdot\frac{(p-1)(N-2)-\alpha}{p} \,|x|^{\alpha-2}+\eta |x|^\beta\right) dx\\
	&\quad +\bigg(\frac{4\alpha(p-1)}{p^2}-\frac{p-1}{p^2}(4+2Np-4N)\bigg)\int_{\R^N}V_\eps^p|u|^p |x|^\alpha dx\\
&\ge \beta_0\int_{\R^N}V_\eps^{p}|u|^pdx+ m \int_{\R^N}V_\eps^{p-1}|u|^p\, dx+\delta_\alpha\int_{\R^N}V_\eps^p|u|^p |x|^\alpha dx\,.
\end{align*}
Thus the proof of the lemma is concluded.
\end{proof}

Applying Corollary \ref{diss}, Lemma \ref{main-lem} and Theorem \ref{Oka-th1.7} we obtain the following generation results. We distinguish the two cases $\alpha\le 2$ and $\alpha>2$ since the hypotheses on the unperturbed operator $L$ are different.
\begin{theorem}\label{th1_1}
Assume $0\le \alpha \le 2$. Set $k=\min\{\beta_0,(p-1)\gamma_0\}$. If $2p<N$ and $ \alpha \le (N-2)(p-1)$ then, for every $c<k$ the operator $ L+\frac{c}{|x|^2}$ endowed with the domain $D_p$ defined in Theorem \ref{thm-luca-simona}
generates a contractive positive $C_0$-semigroup in $L^p(\R^N)$. Moreover, $C_c^\infty(\R^N)$ is a core for such an operator.
\noindent
Finally, the closure of $\left(L+\frac{k}{|x|^2}, D_p\right)$  generates a contractive positive $C_0$-semigroup in $L^p(\R^N)$.
\end{theorem}
\begin{theorem}\label{th1_2}
Assume $\alpha >2$. Set $k=\min\{\beta_0,(p-1)\gamma_0\}$. If $\frac{N}{N-2}<p<\frac{N}2$ and $\alpha <\frac{N(p-1)}{p}$, then for every $c<k$ the operator $ L+\frac{c}{|x|^2}$ endowed with the domain $\widehat{D_p}$ given in Theorem \ref{thm-giorgio-chiara} generates a contractive positive $C_0$-semigroup in $L^p(\R^N)$. Moreover, $C_c^\infty(\R^N)$ is a core for such an operator.
\noindent
Finally, the closure of $\left(L+\frac{k}{|x|^2}, \widehat{D_p}\right)$  generates a contractive positive $C_0$-semigroup in $L^p(\R^N)$.
\end{theorem}

\noindent
The proofs of the two above theorems are identical. We limit ourselves in proving the latter.
\begin{proof}[Proof of Theorem \ref{th1_2}]
In order to apply Theorem \ref{Oka-th1.7}, set $A=-L$, $D(A)=\widehat{D_p}$, $D=C_c^\infty(\R^N)$ and let $B$ be the multiplicative operator by $\frac{1}{|x|^2}$ endowed with the maximal domain $D(|x|^{-2})=\{u\in L^p(\R^N);\,|x|^{-2}u\in L^p(\R^N)\}$
in $L^p(\R^N)$. We observe that the Yosida approximation $B_\eps$ of $B$ is the multiplicative operator by $V_\eps=\frac{1}{|x|^2+\eps}$. Both $A$ and $B$ are $m$--accretive in $L^p(\R^N)$. Then, Lemma \ref{main-lem} yields $(i)$ in Theorem \ref{Oka-th1.7} with $k_1=\beta_0$, $\tilde{c}=0$ and $a=0$. The second assumption $(ii)$ in Theorem \ref{Oka-th1.7} is obviously satisfied. The last one, $(iii)$, holds with $k_2=(p-1)\gamma_0$ thanks to Corollary \ref{diss}. Then, we infer that for every $c <k$, $-L-\frac{c}{|x|^2}$ with domain $\widehat{D_p}$ is $m$--accretive in $L^p(\R^N)$ and $C_c^\infty(\R^N)$ is a core for $-L- \frac{c}{|x|^2}$ by Theorem \ref{thm-giorgio-chiara}. Moreover, $-L-\frac{k}{|x|^2}$ is essentially $m$--accretive. By the Lumer Phillips Theorem (cf. \cite[Chap.II, Theorem 3.15]{EN}) we obtain the generation result. Finally, the positivity of the semigroup is a consequence of Proposition \ref{dispersivity}. The dispersivity is equivalent to the positivity of the resolvent, which is equivalent to the positivity of the semigroup.
\end{proof}


If $2p\ge N$, then $\beta_0\le 0$ and we cannot apply Theorem \ref{Oka-th1.7}. However, if at least $\beta_\alpha\ge 0$, that is $2p-N\le \alpha$, then we still have a generation result, relying on the following abstract theorem by Okazawa (see \cite[Theorem 1.6]{Oka96}).

\begin{theorem} \label{Oka-th1.6}
Let $A$ and $B$ be linear $m$--accretive operators in $L^p(\R^N)$, $1<p<+\infty$. Let $D$ be a core of $A$. Assume that there are constants $\tilde{c}, a, b\ge 0$ such that for all $u\in D$ and $\eps >0$,
$$
{\mathcal Re}\langle Au, \|B_\eps u\|_p^{2-p}|B_\eps u|^{p-2}B_\eps u\rangle \ge -b\|B_\eps u\|_p^2-\tilde{c}\|u\|_p^2-a\|B_\eps u\|_p\|u\|_p,
$$
where $B_\eps :=B(I+\eps B)^{-1}$ denotes the Yosida approximation of $B$. If $\nu >b$ then $A+\nu B$ with domain $D(A)\cap D(B)$ is $m$--accretive and $D(A)\cap D(B)$ is core for $A$. Moreover, $A+bB$ is essentially $m$--accretive on $D(A)\cap D(B)$.
\end{theorem}

In our framework the above result leads to the following theorems. We recall that $D(|x|^{-2})=\{u\in L^p(\R^N);\,|x|^{-2}u\in L^p(\R^N)\}$.
\begin{theorem}\label{th1_1bis}
Assume $0\le \alpha \le 2$. If $2p\ge N$ and $ 2p-N\le \alpha \le (N-2)(p-1)$ then, for every $c<\beta_0$ the operator $ L+\frac{c}{|x|^2}$ endowed with the domain $D_p\cap D(|x|^{-2})$, where $D_p$ is defined in Theorem \ref{thm-luca-simona},
generates a contractive analytic $C_0$-semigroup in $L^p(\R^N)$.
\noindent
Moreover, the closure of $\left(L+\frac{\beta_0}{|x|^2}, D_p\cap D(|x|^{-2})\right)$  generates a contractive analytic $C_0$-semigroup in $L^p(\R^N)$.
\end{theorem}
\begin{theorem} \label{th1_2bis}
Assume $\alpha >2$. If $2p\ge N$ and $ 2p-N\le \alpha <\frac{N(p-1)}{p}$, then for every $c<\beta_0$ the operator $ L+\frac{c}{|x|^2}$ endowed with the domain $\widehat{D_p}\cap D(|x|^{-2})$, where $\widehat{D_p}$ is given in Theorem \ref{thm-giorgio-chiara}, generates a contractive analytic $C_0$-semigroup in $L^p(\R^N)$. Moreover, the closure of $\left(L+\frac{\beta_0}{|x|^2}, \widehat{D_p}\cap D(|x|^{-2})\right)$  generates a contractive analytic $C_0$-semigroup in $L^p(\R^N)$.
\end{theorem}

As before, we limit ourselves in proving the latter.
\begin{proof}[Proof of Theorem \ref{th1_2bis}]
In order to apply Theorem \ref{Oka-th1.6}, set $A=-L$, $D(A)=\widehat{D_p}$, $D=C_c^\infty(\R^N)$ and let $B$ be the multiplicative operator by $\frac{1}{|x|^2}$ endowed with the maximal domain $D(|x|^{-2})$ in $L^p(\R^N)$. Both $A$ and $B$ are $m$--accretive in $L^p(\R^N)$. Then, Lemma \ref{main-lem} and Theorem \ref{Oka-th1.6} (with $b=-\beta_0$, $\tilde{c}=0$ and $a=0$) imply that $ \Big(L+\frac{c}{|x|^2},\widehat{D_p}\cap D(|x|^{-2})\Big)$ is $m$--accretive in $L^p(\R^N)$ for any $c<\beta_0$ and is essentially $m$--accretive if $c=\beta_0$. From the assumptions  $2<\alpha <\frac{N(p-1)}{p}$ it follows that $p>N/(N-2)$ and this yields $\alpha <(N-2)(p-1)$. Therefore,  by Theorem \ref{thm-giorgio-chiara}, $L$ generates a positive $C_0$-semigroup of contractions, which is also analytic. By inspecting the proof of \cite[Theorem 8.1]{Met_Spi} it turns out that there exists $\ell_\alpha>0$ such that
$$
|{\mathcal Im}\langle Lu, |u|^{p-2}u\rangle|\le \ell_\alpha\,    \left(-{\mathcal Re}\langle Lu, |u|^{p-2}u\rangle\right)
$$
for every $u\in  \widehat{D_p}$ (the computations can be performed for $u\in C_c^\infty(\R^N)$ and then one get the estimate for $u\in \widehat{D_p}$ using the fact that $C_c^\infty(\R^N)$ is a core for $L$). Now, the previous estimate continues to hold for all $u\in \widehat{D_p}\cap D(|x|^{-2})$
replacing $L$ with $L+\frac{c}{|x|^2}$, $c\le \beta_0$. This implies that $e^{\pm i\theta} \left(L+\frac{c}{|x|^2}\right)$ is dissipative, where $\cot\theta=\ell_\alpha$. By \cite[Theorem 4.6, Chapter 2]{EN}, it follows that $L+\frac{c}{|x|^2}$ is sectorial and hence generates an analytic semigroup in $L^p(\R^N)$. This ends the proof.
\end{proof}

If we consider the operator $\widetilde{L}$ instead of $L$ the above conditions on $p$ can be simplified. So, by Theorem \ref{thm:CRT}, Proposition \ref{p:tilde} and Lemma \ref{l:tilde}, we can apply Theorem \ref{Oka-th1.7} (Theorem \ref{Oka-th1.6}, respectively) since $\delta_\alpha \ge 0$ if and only if $\alpha \ge 1+\frac{N}{2}(p-2)$.
\begin{theorem}
Assume  $\beta >\alpha -2>0$ and $\eta>0$. Set $k=\min\{\beta_0,(p-1)\gamma_0\}$. If $\alpha \ge 1+\frac{N}{2}(p-2)$ and $N>2p$ then for every $c<k$, the operator $\widetilde{L}+\frac{c}{|x|^2}$ endowed with the domain $\widetilde{D_p}$ given in Theorem \ref{thm:CRT} generates a positive and quasi-contractive $C_0$-semigroup in $L^p(\R^N)$. Moreover, $C_c^\infty(\R^N)$ is a core for such an operator.
\noindent
Finally, the closure of $\left(\widetilde{L}+\frac{k}{|x|^2}, \widetilde{D_p}\right)$  generates a positive and quasi-contractive $C_0$-semigroup in $L^p(\R^N)$.
\end{theorem}

\begin{theorem}
Assume  $\beta >\alpha -2>0$ and $\eta>0$. If $\alpha \ge 1+\frac{N}{2}(p-2)$ and $N\le 2p$ then for every $c<\beta_0$, the operator $\widetilde{L}+\frac{c}{|x|^2}$ endowed with the domain $\widetilde{D_p}\cap D(|x|^{-2})$ generates a quasi-contractive $C_0$-semigroup in $L^p(\R^N)$. Moreover, the closure of $\left(\widetilde{L}+\frac{\beta_0}{|x|^2}, \widetilde{D_p}\cap D(|x|^{-2})\right)$  generates a quasi-contractive $C_0$-semigroup in $L^p(\R^N)$.
\end{theorem}

Let us end with the study of the optimality of the constant $\beta_0$ in \eqref{main-inequality}.
\begin{proposition}
Assume that
\begin{equation}        \label{ipo-C}
{\mathcal Re}\langle -Lu, |V u|^{p-2}V u\rangle \ge C\|Vu\|^p_p,
\end{equation}
for some $C>0$, where $V=\frac 1 {|x|^2}$ and $\alpha \in \N $. Then, $C\le \beta_0$.
\end{proposition}
\begin{proof}
Take $u(x)=v(r)\ge 0$, $r=|x|$. Then
\begin{align*}
{\mathcal Re}\langle -Lu, |V u|^{p-2}V u\rangle &=-\omega_N\int_0^{+\infty}(1+r^\alpha)\left(v''+\frac{N-1}{r}v'\right) r^{-2(p-1)}v^{p-1}r^{N-1}dr\\
&=J,
\end{align*}
where $\omega_N$ denotes the measure of the unit ball in $\R^N$. Choose $v(r)=r^{\beta}e^{-r/p}$, with $\beta>\frac{2p-N}{p}$. Then
\begin{align*}
J&=-\omega_N\int_0^{+\infty}(1+r^\alpha)\left(\beta(\beta+N-2)r^{\delta-1}+\frac{1-N-2\beta}{p}r^\delta+\frac 1 {p^2}r^{\delta+1} \right)e^{-r}dr,
\end{align*}
where we have set $\delta=\beta p+N-2p$. Notice that $\delta>0$ thanks to the choice of $\beta$. Using the properties of the Euler Gamma function, we have
\begin{align*}
J&=-\omega_N \left(\beta(\beta+N-2)+\frac{1-N-2\beta}{p}\delta+\frac 1 {p^2}\delta(\delta+1) \right)\Gamma(\delta)\\
&\quad -\omega_N\left(\beta(\beta+N-2)+\frac{1-N-2\beta}{p}(\delta+\alpha)+\frac 1 {p^2}(\delta+\alpha)(\delta+\alpha+1) \right)\Gamma(\delta+\alpha).
\end{align*}
Now, observe that $\|Vu\|_p^p=\omega_N \Gamma(\delta)$. Hence from \eqref{ipo-C} it follows that
\begin{align*}
C \,\Gamma(\delta)& \le-\left(\beta(\beta+N-2)+\frac{1-N-2\beta}{p}\delta+\frac 1 {p^2}\delta(\delta+1) \right)\Gamma(\delta)\\
&\quad -\left(\beta(\beta+N-2)+\frac{1-N-2\beta}{p}(\delta+\alpha)+\frac 1 {p^2}(\delta+\alpha)(\delta+\alpha+1) \right)\Gamma(\delta+\alpha).
\end{align*}
If $\alpha=n\in \N$ then $\Gamma(\delta+n)=(\delta+n-1)\cdots \delta \Gamma(\delta)$ and the previous estimate yields
\begin{align*}
C & \le-\left(\beta(\beta+N-2)+\frac{1-N-2\beta}{p}\delta+\frac 1 {p^2}\delta(\delta+1) \right)\\
&\quad -\left(\beta(\beta\!+\!N-2)\!+\!\frac{1-N-2\beta}{p}(\delta\!+\!n)\!+\!\frac 1 {p^2}(\delta\!+\!n)(\delta\!+\!n\!+\!1) \right)(\delta\!+\!n-1)\cdots \delta.
\end{align*}
Letting $\delta\to 0^+$ which corresponds to $\beta\to \frac{2p-N}{p}$ eventually implies
$$
C\le \frac{N(p-1)(N-2p)}{p^2}.
$$
Hence $\beta_0$ is the best constant for \eqref{ipo-C} to hold in the case $\alpha\in\N$.
\end{proof}

\section*{Acknowledgments} We are grateful to the referee for valuable comments, suggestions and substantial improvements to the paper.


\begin{thebibliography}{99}

\bibitem{CRT}
\newblock A. Canale, A. Rhandi and C. Tacelli,
\newblock Schr\"odinger type operators with unbounded diffusion and potential terms,
\newblock \emph{Ann. Scuola Norm. Sup. Pisa Cl. Sci.}, to appear, Available on ArXiv (http://arxiv.org/abs/1406.0316v1).

\bibitem{BG}
\newblock P. Baras and  J. A. Goldstein,
\newblock The heat equation with a singular potential,
\newblock \emph{Trans. Am. Math. Soc.}, \textbf{284} (1984), 121--139.

\bibitem{DR}
\newblock T. Durante and  A. Rhandi,
\newblock On the essential self-adjointness of Ornstein-Uhlenbeck operators perturbed by inverse-square potentials,
\newblock \emph{Discrete Cont. Dyn. Syst. S.}, \textbf{6} (2013), 649--655.

\bibitem{EE}
\newblock D. E. Edmunds and  W. E. Evans,
\newblock \emph{Spectral Theory and Differential Operators}, 
\newblock Clarendon Press, Oxford, 1987.

\bibitem{EN}
\newblock K. J. Engel and  R. Nagel,
\newblock \emph{One-Parameter Semigroups for Linear Evolution Equations}, 
\newblock Springer-Verlag, New York, 2000.

\bibitem{FL}
\newblock S. Fornaro and  L. Lorenzi,
\newblock Generation results for elliptic operators with unbounded diffusion coefficients in $L^p$- and $C_b$-spaces,
\newblock \emph{Discrete Contin. Dyn. Syst.}, \textbf{18} (2007), 747--772.


\bibitem{FR}
\newblock S. Fornaro and  A. Rhandi,
\newblock On the Ornstein Uhlenbeck operator perturbed by singular potentials in {$L^p$}-spaces,
\newblock \emph{Discrete Contin. Dyn. Syst.}, \textbf{33} (2013), 5049--5058.

\bibitem{GT}
\newblock D. Gilbarg and  N. Trudinger,
\newblock \emph{Elliptic Partial Differential Equations of Second Order}, 
\newblock Springer, 1983.

\bibitem{LR}
\newblock L. Lorenzi and  A. Rhandi,
\newblock On Schr\"odinger type operators with unbounded coefficients: generation and heat kernel estimates,
\newblock \emph{J. Evol. Equ.}, \textbf{15} (2015), 53--88.

\bibitem{MS}
\newblock G. Metafune and  C. Spina,
\newblock An integration by parts formula in Sobolev spaces,
\newblock \emph{Mediterranean Journal of Mathematics}, \textbf{5} (2008), 357--369.

\bibitem{Met_Spi}
\newblock G. Metafune and  C. Spina,
\newblock Elliptic operators with unbounded diffusion coefficients in $L^p$ spaces,
\newblock \emph{Ann. Scuola Norm. Sup. Pisa Cl. Sci.}, \textbf{XI} (2012), 303--340.

\bibitem{MOSS}
\newblock G. Metafune, N. Okazawa, M. Sobajima and  C. Spina,
\newblock Scale invariant elliptic operators with singular coefficients,
\newblock \emph{J. Evol. Equ.}, \textbf{16} (2016), 391--439.

\bibitem{Mi}
\newblock E. Mitidieri,
\newblock A simple approach to Hardy inequalities,
\newblock \emph{Mat. Zametki}, \textbf{67}   (2000), 563--572.

\bibitem{Na86}
\newblock  R. Nagel (ed.),
\newblock \emph{One-Parameter Semigroups of Positive Operators}, 
\newblock Lecture Notes in Math. \textbf{1184}, Springer-Verlag, 1986.

\bibitem{Oka82}
\newblock N. Okazawa,
\newblock On the perturbation of linear operators in Banach and Hilbert spaces,
\newblock \emph{J. Math. Soc. Japan}, \textbf{34} (1982), 677--701.

\bibitem{Oka96}
\newblock N. Okazawa,
\newblock $L^p$-theory of Schr\"odinger operators with strongly singular potentials,
\newblock \emph{Japan. J. Math.}, \textbf{22} (1996), 199--239.

\bibitem{Ouh}
\newblock E. M. Ouhabaz,
\newblock \emph{Analysis of Heat Equations on Domains}, 
\newblock London Math. Soc. Monographs \textbf{31}, Princeton Univ. Press 2004.

\bibitem{RS}
\newblock M. Reed and  B. Simon,
\newblock \emph{Methods of Modern Mathematical Physics II: Fourier Analysis, Self-Adjointness}, 
\newblock Academic Press, New York, 1975.

\bibitem{S}
\newblock B. Simon,
\newblock Essential self-adjointness of Schr\"odinger operators with singular potentials,
\newblock \emph{Arch. Rational Mech. Anal.}, \textbf{52} (1973), 44--48.

\end{thebibliography}
\end{document}